\documentclass{article}
\usepackage{geometry, amsmath, amsfonts, amsthm, amssymb,enumitem}
\usepackage{tikz, scalerel}
\usepackage{hyperref}
\usepackage[capitalize,noabbrev]{cleveref}

\newtheorem{theorem}{Theorem}[section]
\newtheorem{corollary}[theorem]{Corollary}
\newtheorem{proposition}[theorem]{Proposition}
\newtheorem{lemma}[theorem]{Lemma}
\theoremstyle{definition}
\newtheorem{definition}[theorem]{Definition}
\newtheorem{example}[theorem]{Example}

\newtheorem{question}[theorem]{Question}

\newcommand*{\lowerparen}[2]{%
  \raisebox{-#1}{\(\left(\raisebox{#1}{\(\displaystyle #2 \)}\right)\)}}
\newcommand*{\letteroverline}[1]{\mspace{4mu}\overline{\mspace{-4mu}#1}}

\newcommand*{\sigmam}{\sigma_{\!\scaleobj{0.8}{-}}}

\tikzset{every node/.style={draw, circle, inner sep=0.3mm, minimum size=1em}}

\title{\vspace{-2em} An inverse problem for the collapsing sum}
\author{Travis Dillon}

\begin{document}

\maketitle

\begin{abstract}
Gaussian filters have applications in a variety of areas in computer science, from computer vision to speech recognition. The collapsing sum is a matrix operator that was recently introduced to study Gaussian filters combinatorially. In this paper, we view the collapsing sum from a discrete tomographical perspective and examine the recoverability of its preimages as a matrix completion problem. Using bipartite graphs, we derive a necessary and sufficient condition for a partial matrix to be extended to a preimage of a given matrix.
\end{abstract}

\section{Introduction}

Gaussian filters play a central role in image and signal processing, with applications in human and computer vision, edge detection, and speech processing \cite{gaussian-convolution-alg-survey, recursive-gaussian-blur}. The \textit{collapsing sum} is a matrix operator introduced in \cite{mjum-gauss-blur} to study Gaussian filters combinatorially.

Let $A$ be an $m\times n$ matrix. We denote the entries of a matrix by the corresponding lower-case letter, so the $(i,j)$th entry of $A$ is $a_{i,j}$. The collapsing sum of $A$ is the $(m-1) \times (n-1)$ matrix $\sigma(A)$ with entries 
\[ \sigma(A)_{i,j} = a_{i,j} + a_{i+1,j} + a_{i,j+1} + a_{i+1,j+1}. \]
When scaled by a factor of $1/4$, the collapsing sum returns an average of nearby entries. Applying the operator multiple times averages a matrix over larger regions in the same manner as a Gaussian filter. To slightly simplify the mechanics of later calculations, we work with the \textit{balanced collapsing sum} $\sigmam$, whose entries are given by
\[ \sigmam(A)_{i,j} = a_{i,j} - a_{i+1,j} - a_{i,j+1} + a_{i+1,j+1}. \]
We also state the analogues of our results for the collapsing sum, which follow from straightforward modifications to our proofs.

The overarching philosophy of discrete tomography is to reconstruct discrete structures from a small number of projections (see \cite{discrete-tomography,discrete-tomography2} for an overview of the field). Its origins can be traced back to the 1950s, when Gale \cite{gale-network-flows} and Ryser \cite{ryser-combinatorial-properties} provided necessary and sufficient conditions for a binary matrix to have specified row and columns sums. However, each pair of projections usually corresponds to multiple matrices, so some \textit{a priori} information about the structure of the matrix, such as convexity or periodicity, is usually assumed. Much of the classical work in discrete tomography focused on linear projections. Nivat \cite{plane-tilings} introduced \textit{rectangular scans}, a different kind of matrix projection, and studied their tomographical properties. Further research on rectangular scans has been conducted by Frosini, Nivat, and Rinaldi \cite{binary-matrix-reconstruction,two-rectangular-windows,01-equivalence}, among others.

The main problem in this paper continues in this vein. The collapsing sum can be viewed as a $2\times 2$ rectangular scan, and our goal, broadly defined, is to reconstruct a matrix from its (balanced) collapsing sum. In contrast to the work in \cite{binary-matrix-reconstruction,two-rectangular-windows,plane-tilings,01-equivalence}, our tools will be algebraic as well as combinatorial. Stolk and Batenburg \cite{algebraic-discrete-tomography} have described an algebraic framework for discrete tomography, and the collapsing sum is particularly amenable to linear algebraic methods \cite{mjum-gauss-blur}.

In general, we fix an arbitrary additive abelian group $G$ and denote by $G^{m\times n}$ the additive group of $m\times n$ matrices with entries in $G$. Although technically speaking multiplication is not defined for elements of $G$, we define $1 \cdot g = g \cdot 1 = g$ for every $g\in G$ and set $e_n$ as the $n\times 1$ vector in which every entry is $1$. In \cref{sec:kernel}, we prove the following isomorphism using the balanced collapsing sum.

\begin{theorem}\label{thm:sigmam-iso}
Let $\mathcal{K}_{m,n} = \{ ue_n^T + e_mv^T : u \in G^{m\times 1} \text{ and } v \in G^{n\times 1} \}$. If $m,n \geq 2,$ then $G^{(m-1)\times(n-1)} \cong G^{m\times n}/\mathcal{K}_{m,n}$.
\end{theorem}

A similar result holds for the collapsing sum.\footnote{If $f_n$ is the $n\times 1$ vector whose $i$th coordinate is $(-1)^i$, then
\[G^{(m-1)\times(n-1)} \cong G^{m\times n} / \{uf_n^T + f_mv^T : u \in G^{m\times 1} \text{ and } v \in G^{n \times 1} \}.\]} The set $\mathcal{K}_{m\times n}$ is the kernel of the map $\sigmam\colon G^{m\times n} \to G^{(m-1)\times(n-1)}$. When $G$ is an ordered group ($\mathbb{Z}$ or $\mathbb{R}$, say), this set can be described in another manner. A matrix $A \in G^{m\times n}$ is a \textit{Monge matrix} if $a_{i,k} + a_{j,l} \leq a_{i,l} + a_{j,k}$ for every $1 \leq i < j \leq m$ and $1 \leq k < l \leq n$; it is an \textit{anti-Monge matrix} if the inequality is reversed. It turns out that $A$ is a Monge matrix if and only if $a_{i,k} + a_{i+1,k+1} \leq a_{i+1,k} + a_{i,k+1}$ for every $1 \leq i < m$ and $1 \leq k < n$. In other words, $A$ is a Monge matrix if and only if $\sigmam(A) \leq 0$ and an anti-Monge matrix if and only if $\sigmam(A) \geq 0$ (the inequalities are entrywise). Thus, the kernel of $\sigmam$ consists of exactly those matrices that are simultaneously Monge and anti-Monge (such matrices are sometimes referred to as \textit{sum matrices}). Monge and anti-Monge matrices are of particular interest in combinatorial optimization; see \cite{monge-matrices-optimization} for a survey.

In terms of discrete tomography, \cref{thm:sigmam-iso} means that every collapsing sum reconstruction problem is solvable. In fact, the equation $\sigmam(X) = B$ has the same number of solution matrices for every choice of matrix $B \in G^{(m-1) \times (n-1)}$. In order to make further progress, then, we have to assume some \textit{a priori} information about $X$. One way to do this is with a partial matrix.


Let $\ast$ be a symbol not in $G$. (We think of $\ast$ as a blank entry.) A \textit{partial matrix} is a matrix with entries in $G \cup \{\ast\}$, and the set of $m\times n$ partial matrices is denoted by $G_*^{m\times n}$. A \textit{completion} of a partial matrix $A \in G_*^{m\times n}$ is a matrix $C \in G^{m\times n}$ so that $c_{i,j} = a_{i,j}$ whenever $a_{i,j} \neq \ast$. An $m\times n$ partial matrix $A$ is \textit{consistent} with an $(m-1)\times (n-1)$ matrix $B$ if $A$ has a completion whose balanced collapsing sum is $B$. If $A$ and $B$ are two partial matrices with the same dimensions, we let $(A+B)_{i,j} = a_{i,j} + b_{i,j}$ if $a_{i,j} \neq *$ and $b_{i,j} \neq *$, and $(A+B)_{i,j} = *$ otherwise. Our tomographical problem can be formally stated as follows.

\begin{question}\label{question:partial-completion}
Given a partial matrix $A \in G_*^{m\times n}$ and a matrix $B \in G^{(m-1)\times(n-1)}$, under what conditions is $A$ consistent with $B$? If $A$ is consistent with $B$, when does it have a unique completion $C$ so that $\sigmam(C) = B$?
\end{question}

If every entry of $A$ is $\ast$, then \cref{question:partial-completion} reduces to the case with no \textit{a priori} information. Similar questions on matrix completion have been asked in different contexts. Each problem seeks conditions on a partial matrix that guarantee a completion with a specific property. Some look for completions with a prescribed spectrum or characteristic polynomial \cite{matrix-completion-eigenvalues}, while others look for completions that are positive (semi)definite or a have a specified rank \cite{matrix-completion-survey}.

To address \cref{question:partial-completion}, we introduce some terminology. For each partial matrix $A \in G_*^{m\times n}$, we define the bipartite graph $H_A$ on the bipartition $\big[\{x_i : i \in [m] \},\ \{y_i : i \in [n]\}\big]$ with the edge $(x_i,y_j)$ if $a_{i,j} \neq \ast$ (see Figure \ref{fig:bipartite-graph-example}). A sequence $c = (i_1,j_1,i_2,j_2, \dots, j_k, i_1)$ is a \textit{cycle in $A$} if $w(c) = (x_{i_1},y_{j_1}, \dots, y_{j_k}, x_{i_1})$ is a cycle in $H_{A}$, in other words, if $w(c)$ is a closed walk in $H_A$ that does not repeat edges. A cycle $c$ is \textit{minimal} if the subgraph of $H_A$ induced by $\{x_{i_1},y_{j_1}, \dots, y_{j_k}\}$ does not contain a cycle on fewer vertices.  A cycle $(i_1,j_1,i_2,j_2, \dots, j_k, i_{k+1} = i_1)$ in $A$ is \textit{balanced} if $\sum_{r=1}^k (a_{i_r,j_r} - a_{i_{r+1},j_r}) = 0$, and $A$ is \textit{cycle-balanced} if every cycle in $A$ is balanced. (If $A$ contains no cycles, then it is trivially cycle-balanced.) It turns out (see \cref{thm:minimal-cycles-balanced}) that we need only check the minimal cycles to verify that $A$ is cycle-balanced. We denote the edge set of $H$ by $E(H)$ and the number of connected components of $H$ by $c(H)$. Our second theorem answers \cref{question:partial-completion} and counts the exact number of completions of $A$ that collapse to $B$.

\begin{figure}\label{fig:bipartite-graph-example}
\centering
\begin{minipage}{0.35\textwidth}
\[ \begin{pmatrix}
    3 & 0 & * & * \\
    8 & * & 2 & 0 \\
    * & 1 & * & 7
\end{pmatrix} \]
\end{minipage}
\begin{minipage}{0.35\textwidth}
\begin{center}\begin{tikzpicture}
\foreach \label/\x in {3/1, 2/0, 1/-1}
    \node (a\label) at (\x,1) {$x_\label$};
\foreach \label/\x in {4/1.5, 3/0.5, 2/-0.5, 1/-1.5}
    \node (b\label) at (\x,0) {$y_\label$};
\foreach \from/\to in {a1/b1, a1/b2, a2/b1, a2/b3, a2/b4, a3/b2, a3/b4}
    \draw (\from) to (\to);
\end{tikzpicture}\end{center}
\end{minipage}
\caption{A partial matrix in $\mathbb{Z}_*^{3\times 4}$ and its corresponding bipartite graph.}
\end{figure}
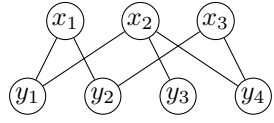

\begin{theorem}\label{thm:bipart-recovery}
Let $G$ be an additive abelian group, $A$ be a partial matrix in $G_\ast^{m\times n}$, and $\widetilde B$ be any matrix in the preimage of $B \in G^{(m-1)\times (n-1)}$ under $\sigmam$. The partial matrix $A$ is consistent with $B$ if and only if $A - \widetilde{B}$ is cycle-balanced; in this case, $A$ has a unique completion $C$ satisfying $\sigmam(C) = B$ if and only if $H_A$ is connected. Moreover, if $\lvert G \rvert = k$ and $A$ is consistent with $B$, then there are exactly $k^{c(H_{A})-1}$ completions of $A$ that are preimages of $B$ under $\sigmam$.
\end{theorem}

A similar statement is true for the collapsing sum.\footnote{Let $G$ be an additive abelian group, $A$ be a partial matrix in $G_\ast^{m\times n}$, and $\widetilde B$ be any matrix in the preimage of $B \in G^{(m-1)\times (n-1)}$ under $\sigma$. There is a completion of $A$ whose collapsing sum is $B$ if and only if $A' := A - \widetilde{B}$ satisfies $\sum_{r=1}^k (a'_{i_r,j_r} + a'_{j_r,i_{r+1}})=0$ for every cycle $(i_1,j_1,i_2,j_2, \dots, j_k, i_{k+1} = i_1)$ in $A$; in this case, $A$ has a unique completion $C$ satisfying $\sigma(C) = B$ if and only if $H_A$ is connected. Moreover, if $\lvert G \rvert = k$ and $A$ is consistent with $B$, then there are exactly $k^{c(H_{A})-1}$ completions of $A$ that are preimages of $B$ under $\sigma$.}
The proof of \cref{thm:bipart-recovery} that appears in \cref{sec:bipartite} is valid even when $A$ has a countably infinite number of rows or columns.

\section{Kernel of the balanced collapsing sum}\label{sec:kernel}

Given two matrices $A,B \in G^{m\times n}$, we write $A \sim B$ if there are vectors $u \in G^{m\times 1}$ and $v \in G^{n\times 1}$ so that $A = B + ue_n^T + e_mv^T$. In other words, $A \sim B$ if and only if there exist $u_1, \dots, u_m, v_1,\dots, v_n \in G$ so that $a_{i,j} = b_{i,j} + u_i + v_j$ for every $(i,j) \in [m]\times[n]$. The relation $\sim$ is an equivalence and a congruence: If $A \sim B$, then $A + C \sim B+ C$ for any matrix $C \in G^{m\times n}$. The equivalence class of $A$ under $\sim$ is denoted $[A]$, and the equivalence class of $0_{m\times n}$ is denoted $\mathcal{K}_{m\times n}$.

\begin{lemma}\label{thm:sigmam-constant-eq-classses}
The balanced collapsing sum $\sigmam$ is constant on equivalence classes. That is, if $A \sim B$, then $\sigmam(A) = \sigmam(B)$.
\end{lemma}
\begin{proof}
It suffices to prove that $\sigmam(e_mu^T) = \sigmam(ve_n^T) = 0$ for all $u \in G^{n\times 1}$ and $v \in G^{m\times 1}$. A generic element of $\sigmam(e_mu^T)$ is 
\[
\sigmam(e_mu^T)_{p,q}
= \sigmam\begin{pmatrix} u_q & u_{q+1} \\ u_q & u_{q+1} \end{pmatrix}
= 0.
\]
The calculation for $ve_n^T$ is similar.
\end{proof}

It will be useful to have a canonical representative for each equivalence class.

\begin{definition}\label{def:bar-matrix}
Let $A \in G^{m\times n}$. The $m\times n$ matrix $\letteroverline{A}$ is defined by
\[ \overline{a}_{i,j} = a_{i,j} - a_{i,1} - a_{1,j} + a_{1,1}. \]
\end{definition}

Setting $u_i = -a_{i,1}$ and $v_j = -a_{j,1} + a_{1,1}$, we have $\overline{a}_{i,j} = a_{i,j} + u_i + v_j$, which shows that $A \sim \letteroverline{A}$. A straightforward calculation shows that $\letteroverline{A} = \letteroverline{B}$ if $A \sim B$, so $\letteroverline{A}$ is indeed a well-defined representative of $[A]$, the unique element of $[A]$ whose first row and column contain only zeros.

\begin{proposition}\label{thm:kernel-sigmam}
Let $m,n \geq 2$ and $\sigmam\colon G^{m\times n} \rightarrow G^{(m-1) \times (n-1)}$. The kernel of $\sigmam$ is exactly $\mathcal{K}_{m\times n}$.
\end{proposition}
\begin{proof}
Since $\letteroverline{K} = 0$ for every $K \in \mathcal K_{m,n}$, we need to show that $\sigmam(A) = 0$ if and only if $\letteroverline{A} = 0$. \cref{thm:sigmam-constant-eq-classses} implies that $\sigmam(A) = \sigmam(\letteroverline{A})$, and clearly $\letteroverline{A} = 0$ implies $\sigmam(\letteroverline{A}) = 0$. If $\letteroverline{A} \not= 0$, let $\overline{a}_{p,q}$ be a nonzero entry of $\letteroverline{A}$ such that $p+q$ is minimal. Then $p,q > 1$ and $\overline{a}_{p-1,q-1} = \overline{a}_{p-1,q} = \overline{a}_{p,q-1} = 0$, so $\sigmam(\letteroverline{A})_{p-1,q-1} = \overline{a}_{p,q} \not= 0$; therefore $\sigmam(A) \neq 0$.
\end{proof}

We note that Frosini and Nivat \cite[Theorem 7]{binary-matrix-reconstruction} prove \cref{thm:kernel-sigmam} in the special case of binary matrices. We next give a construction to show that $\sigmam$ is surjective.

\begin{definition}
Let $A \in G^{m \times n}$. The $(m+1) \times (n+1)$ matrix $A^+$ is defined by
\[ a^+_{p,q} = \sum_{\substack{i < p \\ j < q}} a_{i,j}, \]
where the empty sum has value 0.
\end{definition}

\begin{example}
If $A = \left(\begin{smallmatrix}
	2 & -1\\
	1 & \phantom{-}3
	\end{smallmatrix}\right) \in \mathbb{Z}^{2\times 2}$, then
	\[ A^+ = \begin{pmatrix}
	0 & 0 & 0\\
	0 & 2 & 1\\
	0 & 3 & 5
	\end{pmatrix}.\tag*{$\lozenge$}\]
\end{example}

\begin{lemma}\label{thm:sigmam-of-plus}
If $m,n \geq 2$ and $A \in G^{m\times n}$, then $\sigmam(A^+) = A$.
\end{lemma}
\begin{proof}
The proof is straightforward calculation:
\begin{align*}
\sigmam(A^+)_{p,q}
    &= \lowerparen{0.5em}{\sum_{\substack{i < p\\j < q}} a_{i,j}- \sum_{\substack{i < p+1\\j < q}} a_{i,j}} - \lowerparen{0.5em}{\sum_{\substack{i < p\\j < q+1}} a_{i,j} - \sum_{\substack{i < p+1\\j < q+1}}a_{i,j}}\\
    &= \sum_{j < q} a_{p,j} - \sum_{j < q+1} a_{p,j}\\
    &= a_{p,q}. \qedhere
\end{align*}
\end{proof}

With this, the proof of the first theorem is swift.

\begin{proof}[Proof of \cref{thm:sigmam-iso}]
The map $\sigmam\colon G^{m\times n} \to G^{(m-1) \times (n-1)}$ is linear and therefore a homomorphism. \cref{thm:sigmam-of-plus} shows that $\sigmam$ is surjective. The kernel of $\sigmam$ is $\mathcal{K}_{m,n}$ by \cref{thm:kernel-sigmam}, so applying the Fundamental Homomorphism Theorem finishes the proof.
\end{proof}

Thus, each equivalence class in $G^{m\times n}$ is the preimage of exactly one element of $G^{(m-1) \times (n-1)}$; namely, $[A^+]$ is the preimage of $A$. We can use \cref{thm:sigmam-iso} to count the number and size of equivalence classes of $G^{m\times n}$ for finite groups $G$.

\begin{corollary}\label{thm:counting equivalence classes}
Let $G$ be a finite abelian group of order $k$. There are exactly $k^{(m-1)(n-1)}$ equivalence classes in $G^{m\times n}$, each of size $k^{m+n-1}$.
\end{corollary}

In particular, for every $B \in G^{(m-1)\times(n-1)}$, there are exactly $k^{m+n-1}$ matrices $A \in G^{m\times n}$ such that $\sigmam(A) = B$.

\section{Completion of partial matrices}\label{sec:bipartite}

This section consists mainly of a proof of \cref{thm:bipart-recovery}. To recall the setup: $A$ is an $m\times n$ partial matrix, $B \in G^{(m-1)\times (n-1)}$, and $\widetilde{B}$ is a completion of $A$ such that $\sigmam(\widetilde{B}) = B$. Also, we call the partial matrix $A$ \textit{cycle-balanced} if $\sum_{r=1}^k (a_{i_r,j_r} - a_{i_{r+1},j_r}) = 0$ for every cycle $(x_{i_1},y_{i_1},\dots,x_k,y_k,x_{k+1}=x_1)$ in the bipartite graph $H_A$.

We will need the following two lemmas.

\begin{lemma}\label{thm:minimal-cycles-balanced}
A partial matrix $A$ is cycle-balanced if and only if every minimal cycle in $A$ is balanced.
\end{lemma}
\begin{proof}
If $A$ is cycle-balanced, then every cycle in $A$ is balanced, so in particular the minimal cycles are. Now suppose that every minimal cycle is balanced. We prove that every cycle is balanced by induction. Any 4-cycle is minimal and therefore balanced; now let $2k \geq 6$ and assume that every cycle of length less than $2k$ is balanced. Choose any cycle $\gamma = (i_1,j_1,\dots, i_k,j_k,i_{k+1} = i_1)$ of length $2k$. If $\gamma$ is minimal, then it is balanced. Otherwise there is an edge $(x_{i_s},y_{j_t})$ in $H_A$ with $t\notin \{s-1,s\}$ (mod $k$). We can choose the starting vertex of $\gamma$ so that $s = 1$. Then $(i_1,j_1,\dots,i_t,j_t,i_1)$ and $(i_1,j_t,i_{t+1},\dots,i_k,j_k,i_1)$ are two cycles of length strictly less than $2k$, so both are balanced. We can decompose the sum $\sum_{r=1}^k (a_{i_r,j_r} - a_{i_{r+1}, j_r})$ over these smaller cycles:
\[ \left( \sum_{r=1}^{t-1} (a_{i_r,j_r} - a_{i_{r+1}, j_r})
        + (a_{i_t,j_t} - a_{i_1,j_t}) \right)
    + \left( (a_{i_1,j_t} - a_{i_t,j_t})
        + \sum_{r=t}^{k} (a_{i_r,j_r} - a_{i_{r+1}, j_r}) \right). \]
Since the smaller cycles are balanced, both sums are 0, which shows that $\gamma$ is balanced.
\end{proof}

\begin{lemma}\label{thm:equal-kernel-matrices}
Let $u, \tilde{u} \in G^{m\times 1}$ and $v,\tilde{v} \in G^{n\times 1}$. Then $ue_n^T + e_mv^T =  \tilde{u}e_n^T + e_m\tilde{v}^T$ if and only if $u = \tilde{u} + g e_m$ and $v = \tilde{v} - g e_n$ for some $g \in G$.
\end{lemma}
\begin{proof}
It is straightforward to check that the substitution works for all $g \in G$. For the other direction, we have
\begin{equation}\label{eqn:sum-rank-1-lem}
    u_i + v_j = \tilde{u}_i + \tilde{v}_j
\end{equation}
for all $(i,j) \in [m]\times[n]$. Writing $u_1 = \tilde{u}_1 + g$ for some $g \in G$ and evaluating \eqref{eqn:sum-rank-1-lem} with $i = 1$ gives $v_j= \tilde{v}_j - g$ for all $j \in [n]$; inserting this into \eqref{eqn:sum-rank-1-lem} shows that $u_i = \tilde{u}_i + g$ for all $i \in [m]$.
\end{proof}

\begin{proof}[Proof of \cref{thm:bipart-recovery}]
The theorem consists of solving the equation $\sigmam(X) = B$, where $X$ is a completion of $A$. Since this is equivalent to solving $\sigmam(X - \widetilde{B}) = 0$, we may assume that $B = 0$ by replacing $A$ with the partial matrix $A - \widetilde{B}$. With this substitution, we need only prove the theorem with an arbitrary partial matrix $A$ and $B = \widetilde{B} = 0$.

We first prove that $A$ is consistent with $0$ if and only if $A$ is cycle balanced. To that end, assume that $A$ is consistent with $0$ and let $C$ be a completion of $A$ so that $\sigmam(C) = 0$. Every cycle in $C$ is also a  cycle in $A$, so it suffices to show that $C$ is cycle-balanced. Since $H_C$ is the complete bipartite graph, the minimal cycles in $H_C$ are the 4-cycles. For any $1 \leq i_1 < i_2 \leq m$ and $1 \leq j_1 < j_2 \leq n$, we have
\[ c_{i_1,j_1} - c_{i_2,j_1} + c_{i_2,j_2} - c_{i_1,j_2}
    = \sum_{r=i_1}^{i_2-1} \sum_{k=j_1}^{j_2-1} \sigmam(C)_{r,k}
    =0, \]
since the sum is telescoping. By \cref{thm:minimal-cycles-balanced}, $C$ is cycle-balanced.

To prove the converse, assume that $A$ is cycle-balanced and that $H_A$ has bipartitions $X = \{x_i : i \in [m]\}$ and $Y = \{y_i : i \in [n]\}$; we want to construct a completion $C$ of $A$ that satisfies $\sigmam(C) = 0$. (\cref{ex:completion-construction}, which follows the proof, walks through a particular instance of the following construction.) Let $K_{X,Y}$ be the complete graph with bipartition $[X,Y]$. If $H_A$ is connected, set $A' = A$. Otherwise, let $S \subseteq E(K_{X,Y}) \setminus E(H_A)$ be a subset of edges such that the graph with edge set $E(H_A) \cup S$ is connected and contains no cycles not in $E(H_A)$; we define a new partial matrix $A'$ by setting
\[
a'_{i,j} = \begin{cases}
    a_{i,j} &\text{if } a_{i,j} \neq \ast \\
    0 &\text{if } (x_i,y_j) \in S \\
    \ast &\text{otherwise.}
\end{cases}
\]
The matrix $A'$ is cycle-balanced, since any cycle in $A'$ is a cycle in $A$. For any pair $(s,t)\in [m]\times [n]$, let $(x_s = x_{i_1}, y_{j_1},\dots,x_{i_k},y_{j_k} = y_t)$ be a path in $H_{A'}$. We define the matrix $C$ by
\begin{equation}\label{eq:C-construction}
    c_{s,t} = \sum_{r=1}^k a'_{i_r,j_r} - \sum_{r=1}^{k-1} a'_{i_{r+1},j_r}.
\end{equation}
This sum is independent of the specific path from $i$ to $j$ because $A$ is cycle-balanced. Moreover, $C$ is a cycle-balanced completion of $A$. Taking the cycle $(i,j,i+1,j+1,i)$ in $C$ shows that $\sigmam(C)_{i,j} = 0$.

In the remainder of the proof, we assume that $A$ is consistent with $B$ and that $C$ is a completion of $A$ with $\sigmam(C) = B$. We now prove that $C$ is the unique completion of $A$ satisfying $\sigmam(C) = B$ if and only if $H_A$ is connected. Therefore, let $C'$ be a completion of $A$ such that $\sigmam(C') = B$. By \cref{thm:kernel-sigmam}, there are two vectors $u \in G^{m\times 1}$ and $v \in G^{n\times 1}$ such that $C' - C= ue_n^T + e_mv^T$. Setting $D = C' - C$ gives $d_{i,j} = u_i + v_j$ for all $(i,j) \in [m]\times [n]$. Let $I = \{(i,j) : (x_i,y_j) \in E(H_A)\}$. Since $C'$ and $C$ are both completions of $A$, we have $d_{i,j} = 0$ for each $(i,j) \in I$, which implies that $u_i = -v_j$ for each $(i,j) \in I$. If $x_s$ and $y_t$ are in the same component of $H_A$, there is a path $(x_s = x_{i_1}, y_{j_1}, \dots,x_{j_\ell}, y_{j_\ell} = y_t)$. Then $u_{i_r} = -v_{j_r} = u_{i_{r+1}}$ for each $r$, so $u_s = -v_t$ by induction.

If $H_A$ has only one component, then $u_i = -v_j$ for every $(i,j) \in [m]\times [n]$, so $D = 0$. Therefore $C = C'$, which shows that $A$ has a unique completion in the preimage of $B$. If $H_A$ has at least two components, we may arbitrarily assign an element of $G$ to each component of $H_A$; this uniquely determines the entries of $u$ and $v$. It is straightforward to check that for any such assignment, $C'$ is indeed a completion of $A$ that satisfies $\sigmam(C') = B$. By \cref{thm:equal-kernel-matrices}, changing the value in exactly one component changes the matrix $D$. Therefore $A$ has multiple completions in the preimage of $B$.

If $\lvert G \rvert = k$, then there are $k^{c(H)}$ possible values for $(u,v)$. \cref{thm:equal-kernel-matrices} shows that a given matrix $ue_n^T + e_mv^T$ is produced by exactly $k$ pairs, so there are $k^{c(H)-1}$ completions of $A$ whose balanced collapsing sum is $B$.
\end{proof}

\begin{example}\label{ex:completion-construction}
    Suppose $G = \mathbb{Z}_5$ and set\vspace{-2ex}
    \[ A = 
        \begin{pmatrix}
        0 & * & 1 & * \\
        * & * & * & 2 \\
        1 & * & 2 & *
        \end{pmatrix}.
    \]
    We use the process described in the proof of \cref{thm:bipart-recovery} to construct a completion $C$ of $A$ such that $\sigmam(C) = 0$. The graph $H_A$ is
    
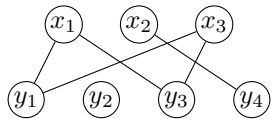
\begin{center}\begin{tikzpicture}
    \foreach \label/\x in {3/1, 2/0, 1/-1}
        \node (a\label) at (\x,1) {$x_\label$};
    \foreach \label/\x in {4/1.5, 3/0.5, 2/-0.5, 1/-1.5}
        \node (b\label) at (\x,0) {$y_\label$};
    \foreach \from/\to in {a1/b1, a1/b3, a2/b4, a3/b1, a3/b3}
        \draw (\from) to (\to);
    \end{tikzpicture}\end{center}
    We note that $A$ is cycle-balanced. We then choose a set of edges to add that (1) results in a connected graph and (2) does not create any new cycles. The set $S = \{(x_2,y_2), (x_2,y_3)\}$ works. Placing zeros in the corresponding entries of $A$ yields the matrix
    \[  A' = 
        \begin{pmatrix}
        0 & * & 1 & * \\
        * & 0 & 0 & 2 \\
        1 & * & 2 & *
        \end{pmatrix}.
    \]
    We then fill in the remaining entries using formula \eqref{eq:C-construction}. For example, $(x_1,y_3,x_2,y_2)$ is a path from $x_1$ to $y_2$, so formula \eqref{eq:C-construction} defines $c_{1,2} = (a_{1,3} + a_{2,2}) - a_{2,3} = 1$. The other entries can be computed in a similar manner to obtain
    \[ C = 
        \begin{pmatrix}
        0 & 1 & 1 & 3 \\
        4 & 0 & 0 & 2 \\
        1 & 2 & 2 & 4
        \end{pmatrix},
    \]
    and it is a simple matter to confirm that $\sigmam(C) = 0$.\hfill$\lozenge$
\end{example}

Although the notion of connectedness featuring in \cref{thm:bipart-recovery} is natural from a graph-theoretic point of view, it can seem unexpectedly finicky from the perspective of matrices, so we conclude with a few comments. It is possible for $H_A$ to be connected when only $m+n-1$ elements of an $m\times n$ partial matrix $A$ are specified (for example, those in the first row and column of $A$); on the other hand, $H_A$ can be disconnected even when $mn-\min\{m,n\}$ elements of $A$ are specified (the first $m-1$ rows or the first $n-1$ columns). Moreover, connectedness is not a stable condition: For each integer $m+n-1 \leq k \leq mn-\min\{m,n\}+1$, there is an $m\times n$ partial matrix $A$ with $k$ specified entries such that $H_A$ is connected, but deleting a single specified entry of $A$ disconnects $H_A$. There do, however, exist partial matrices that are quite stable: If the first $k$ rows and columns of $A$ are specified, then $H_A$ remains connected after converting any $k-1$ known entries to blank ones.\newline

\vspace{1ex}

\noindent
{\large \textbf{Acknowledgments}}\\[0.3em]
The author thanks Samuel Gutekunst for several delightful conversations over the course of this research and the anonymous referees for their helpful comments.

\newpage


\bibliography{bibliography}

\ifx\undefined\bysame
\newcommand{\bysame}{\leavevmode\hbox to3em{\hrulefill}\,}
\fi
\begin{thebibliography}{10}

\bibitem{monge-matrices-optimization}
Rainer Burkard, Bettina Klinz, and R\"udinger Rudolf, {\em Perspectives of
  {M}onge properties in optimization}, Discrete Applied Mathematics {\bf 70}
  (1996), 95--161.

\bibitem{matrix-completion-eigenvalues}
Gl\'oria Cravo, {\em Matrix completion problems}, Linear Algebra and its
  Applications {\bf 430} (2009), 2511--2540.

\bibitem{mjum-gauss-blur}
Travis Dillon, {\em A combinatorial interpretation of {G}aussian blur},
  Minnesota Journal of Undergraduate Mathematics {\bf 5} (2020).

\bibitem{binary-matrix-reconstruction}
Andrea Frosini and Maurice Nivat, {\em Binary matrices under the microscope: A
  tomographical problem}, Theoretical Computer Science {\bf 370} (2007),
  201--217.

\bibitem{two-rectangular-windows}
Andrea Frosini, Maurice Nivat, and Simone Rinaldi, {\em Scanning integer
  matrices by means of two rectangular windows}, Theoretical Computer Science
  {\bf 406} (2008), 90--96.

\bibitem{gale-network-flows}
David Gale, {\em A theorem on flows in networks}, Pacific Journal of
  Mathematics {\bf 7} (1957), 1073--1082.

\bibitem{gaussian-convolution-alg-survey}
Pascal Getreuer, {\em A survey of {G}aussian convolution algorithms}, Image
  Processing On Line {\bf 3} (2013), 286--310.

\bibitem{discrete-tomography}
Gabor Herman and Attila Kuba (eds.), {\em Discrete tomography: Foundations,
  algorithms, and applications}, Birkh\"auser, 1999.

\bibitem{discrete-tomography2}
Gabor Herman and Attila Kuba (eds.), {\em Advances in discrete tomography and
  its applications}, Birkh\"auser, 2007.

\bibitem{matrix-completion-survey}
Monique Laurent, {\em Matrix completion problems}, Encyclopedia of Optimization
  (Christodoulos Floudas and Panos Pardalos, eds.), Springer, 2009,
  pp.~1967--1975.

\bibitem{plane-tilings}
Maurice Nivat, {\em Sous-ensembles homog\`enes de $\mathbb{Z}^2$ et pavages du
  plan}, Comptes Rendus Mathematique {\bf 335} (2002), 83--86.

\bibitem{01-equivalence}
Maurice Nivat, {\em On a tomographic equivalence between $(0,1)$-matrices},
  Theory is Forever (J.~Karhum\"aki, H.~Maurer, G.~P\v{a}un, and G.~Rozenberg,
  eds.), Lecture Notes in Computer Science, vol. 3113, Springer, 2004,
  pp.~216--–234.

\bibitem{ryser-combinatorial-properties}
H.~J. Ryser, {\em Combinatorial properties of matrices of zeros and ones},
  Canadian Journal of Mathematics {\bf 9} (1957), 371--377.

\bibitem{algebraic-discrete-tomography}
Arjen Stolk and K.~Joost Batenburg, {\em An algebraic framework for discrete
  tomography: Revealing the structure of dependencies}, SIAM Journal on
  Discrete Mathematics {\bf 24} (2010), 1056--1079.

\bibitem{recursive-gaussian-blur}
Ian Young and Lucas van Vliet, {\em Recursive implementation of the {G}aussian
  filter}, Signal Processing {\bf 44} (1995), 139--151.

\end{thebibliography}
\bibliographystyle{amsplain-nodash}

\end{document}